\documentclass[12pt]{article}
\usepackage{amsmath,amsfonts,amssymb,amsthm,bm,verbatim}
\usepackage{fullpage, comment,cases}
\usepackage{todonotes}

\newtheorem{thm}{Theorem}
\newtheorem{lem}[thm]{Lemma}

\theoremstyle{definition}
\newtheorem*{Def}{Definition}

\DeclareMathOperator{\girth}{girth}

\newcommand\F{\mathcal{F}}
\newcommand\N{\mathcal{N}}
\newcommand\E{\mathcal{E}}
\newcommand\M{\mathcal{M}}
\newcommand\dist{\text{dist}}
\DeclareMathOperator\im{im}
\begin{document}

\title{Hedetniemi's conjecture is asymptotically false}
\author{Xiaoyu He\thanks{Department of Mathematics, Stanford University, Stanford, CA 94305, USA. Email: {\tt alkjash@stanford.edu}. Research supported by an NSF GRFP grant number DGE-1656518.}\and Yuval Wigderson\thanks{Department of Mathematics, Stanford University, Stanford, CA 94305, USA. Email: {\tt yuvalwig@stanford.edu}. Research supported by an NSF GRFP grant number DGE-1656518.}}

\maketitle
\begin{abstract}
    Extending a recent breakthrough of Shitov, we prove that the chromatic number of the tensor product of two graphs can be a constant factor smaller than the minimum chromatic number of the two graphs. More precisely, we prove that there exists an absolute constant $\delta>0$ such that for all $c$ sufficiently large, there exist graphs $G$ and $H$ with chromatic number at least $(1+\delta)c$ for which $\chi(G \times H) \le c$.
\end{abstract}

\section{Introduction}
If $G$ and $H$ are finite graphs, their \emph{tensor product} $G\times H$ is the graph on $V(G)\times V(H)$ where vertices $(g_1, h_1)$ and $(g_2, h_2)$ are adjacent if and only if $g_1\sim g_2$ and $h_1 \sim h_2$; here and throughout, we use the notation $\sim$ to denote adjacency. By composing with the projection maps to each coordinate, it is easy to check that the chromatic number satisfies
\begin{equation}\label{eq:conj}
    \chi(G \times H) \leq \min\{\chi(G),\chi(H)\},
\end{equation}
for all finite graphs $G$ and $H$. In 1966, Hedetniemi~\cite{He} conjectured that equality always holds in (\ref{eq:conj}). This conjecture has received a considerable amount of attention; for instance, it was proved if $G$ and $H$ are $4$-colorable \cite{ElSa}, if every vertex in $G$ is contained in a large clique \cite{BuErLo}, or if $G$ and $H$ are Kneser graphs or hypergraphs \cite{HaMe}. Additionally, many natural variants of this conjecture have been studied. For instance, Hajnal \cite{Ha} proved that the analogous conjecture is false for infinite graphs, while Zhu \cite{Zh11} proved that the analogous conjecture for fractional colorings is true. We refer the reader to the excellent surveys \cite{Sa,Ta, Zh98} for more information on work surrounding Hedetniemi's conjecture.

In a recent breakthrough, Shitov~\cite{Sh} disproved Hedetniemi's conjecture by demonstrating that for sufficiently large $c$, there exist graphs $G$ and $H$ with $\chi(G)>c,\chi(H)>c$, but $\chi(G \times H) \leq c$. Even more recently, Tardif and Zhu \cite{TaZh} proved that the gap between $\chi(G \times H)$ and $\min \{\chi(G),\chi(H)\}$ can be arbitrarily large, i.e.\ that for every integer $d$ and for all $c$ sufficiently large, there exist graphs $G$ and $H$ with $\chi(G)>c+d,\chi(H)>c+d$, but $\chi(G \times H) \leq c$. They also raised the question of whether the gap $d$ can be made to be linear in $c$, and proved that this is possible under the additional assumption that Stahl's conjecture on multichromatic numbers is true.

In this paper, we modify Shitov's construction to answer Tardif and Zhu's question in the affirmative, showing that the ratio of $\chi(G \times H)$ and $\min\{\chi(G),\chi(H)\}$ is asymptotically bounded away from $1$.

\begin{thm}\label{thm:main}
There is an absolute constant $\delta\geq 10^{-9}$ such that for all sufficiently large $c$, there exist simple graphs $G,H$ with $\chi(G)\geq (1+\delta)c,\chi(H) \geq (1+\delta)c$, and $\chi(G \times H) \leq c$.
\end{thm}
Equivalently, we may express Theorem \ref{thm:main} in terms of the \emph{Poljak--R\"odl function}~\cite{PoRo}, which is defined by
\[
    f(k)=\min_{\chi(G), \chi(H)\geq k}\chi(G \times H).
\]
Then Hedetniemi's conjecture is equivalent to the statement that $f(k)=k$ for all $k$. The so-called weak version of Hedetniemi's conjecture simply asks whether $\lim_{k \to \infty}f(k)=\infty$, and is still open; however, Poljak and R\"odl \cite{PoRo} proved that either $f(k) \to \infty$ or $f(k)$ is bounded by $9$. With this notation, Theorem \ref{thm:main} is equivalent to the statement that
\[
    f(k) \leq (1- \delta)k
\]
for all sufficiently large $k$. For comparison, Shitov \cite{Sh} proved that $f(k)\leq k-1$ for sufficiently large $k$, and Tardif and Zhu \cite{TaZh} proved that $f(k) \leq k-(\log k)^{1/4-o(1)}$. Moreover, they proved that the stronger bound $f(k) \leq (1/2+o(1))k$ would follow from Stahl's conjecture.

The proof of Theorem \ref{thm:main} closely mirrors Shitov's proof of the main result in \cite{Sh}. Roughly speaking, the main innovation in our argument is that whereas Shitov constructs one uncolorable vertex $\nu\in V(H)$ to prove $\chi(H) \ge c+1$, we construct a large clique $\N$ of $\delta c$ uncolorable vertices of $H$ (see Lemma \ref{lem:main}) to obtain the stronger lower bound $\chi(H) \ge (1+\delta)c$.

For the sake of clarity of presentation, we systematically omit floor and ceiling signs whenever they are not crucial.

\section{Definitions and Basic Results}
All graphs are assumed to not have multiple edges, but graphs not specified to be simple may contain loops. We will never refer to the chromatic number of a non-simple graph, as this is not a well-defined integer.

If $H$ is a finite graph, two maps $\phi_1,\phi_2:V(H)\rightarrow[c]$ are called {\it co-proper} if $\phi_1(u)\ne \phi_2(v)$ whenever $u\sim v$ in $H$. The {\it exponential graph} $\E_c(H)$ is the graph on vertex set $[c]^{V(H)}$ where two vertices $\phi_1, \phi_2$ are adjacent if and only if they are co-proper. Observe that a map $\phi$ is co-proper with itself if and only if $\phi$ is a proper coloring of $H$. Therefore, at most one of $H$ and $\E_c(H)$ has loops. Moreover, both will be simple exactly when $H$ is simple and $\chi(H)>c$. To avoid confusion, we will follow the convention of calling vertices of $\E_c(H)$ \emph{maps} and proper colorings of $\E_c(H)$ {\it colorings}.

A pair of the form $(H, \E_c(H))$ is a natural candidate for counterexamples to Hedetniemi's conjecture. Indeed, El-Zahar and Sauer~\cite{ElSa} observed that if $\chi(H\times H') < \min\{\chi(H), \chi(H')\}$ for some $H'$, then this also holds for $H'=\E_c(H)$ where $c=\chi(H)-1$. In particular, they established a simple upper bound on $\chi(H\times \E_c(H))$.

\begin{lem}[El-Zahar and Sauer, \cite{ElSa}]\label{lem:c-colorable}
    For any graph $H$ and any integer $c\geq 1$, 
    \[
        \chi(H \times \E_c(H)) \leq c.
    \]
\end{lem}
% \begin{proof}
%     We construct a coloring $\Phi:V(H \times \E_c(H)) \to [c]$ by declaring
%     \[
%         \Phi(h,\phi)=\phi(h) \in [c]
%     \]
%     for every $h \in V(H)$ and $\phi:V(H) \to [c]$. To see that this is a proper coloring of $H \times \E_c(H)$, suppose that $((h_1,\phi_1),(h_2,\phi_2))$ is an edge in $H \times \E_c(H)$. By the definition of the tensor product, this implies that $(h_1,h_2) \in E(H)$ and that $\phi_1,\phi_2$ are co-proper. By the definition of co-proper maps, this implies that $\phi_1(h_1) \neq \phi_2(h_2)$. Thus, $(h_1,\phi_1)$ and $(h_2,\phi_2)$ receive different colors under $\Phi$, so $\Phi$ is a proper $c$-coloring.
% \end{proof}

If $\Psi$ is a proper $(c+t)$-coloring of an exponential graph $\E_c(H)$, where $t\ge 0$, we call $\{1,\ldots, c\}$ the {\it primary colors} and $\{c+1,\ldots, c+t\}$ the {\it secondary colors}. We say that a proper $(c+t)$-coloring $\Psi$ of $\E_c(H)$ is {\it suited} if for every $\phi \in V(\E_c(H))$, 
\[
\Psi(\phi) \in \im(\phi)\cup \{c+1,\ldots, c+t\}.
\]
In other words, a proper $(c+t)$-coloring is suited if it only assigns a primary color $b$ only to maps $\phi$ which have $b$ in their image.

\begin{lem}\label{lem:suited}
    If $\chi(\E_c(H)) \le c+t$, then $\E_c(H)$ has a suited $(c+t)$-coloring.
\end{lem}
\begin{proof}
    Let $\Psi$ be a proper $(c+t)$-coloring, and let $\phi_i$ be the constant map $v\mapsto i$ in $V(\E_c(H))$. Since the maps $\phi_i$ are pairwise co-proper, they form a clique of size $c$ inside $\E_c(H)$ and $\Psi$ assigns them different colors.
    
    We may permute the colors of $\Psi$ so that $\Psi(\phi_i) = i$ for all $i$. We claim that such a $\Psi$ is suited. Indeed, if $\phi\in V(\E_c(H))$, $\phi$ is co-proper to all constant maps $\phi_i$ where $i$ is a primary color not in $\im(\phi)$. Since this map gets color $i$, we see that $\phi$ is not colored $i$. It follows that $\Psi(\phi) \in \im(\phi)\cup \{c+1,\ldots, c+t\}$ as desired.
\end{proof}

If $G$ is a finite simple graph, write $G^\circ$ for the graph obtained by adding loops to every vertex of $G$. We also write $G\subseteq H$ if $G$ is a subgraph of $H$. Recall that the Erd\H os--Ko--Rado theorem~\cite{ErKoRa} states that if $n\ge 2k$ and every pair of a family of $k$-subsets of an $n$-set intersects, then there are at most $\binom{n-1} {k-1}$ such subsets.

\begin{lem}\label{lem:exp-independence}
    Suppose $H$ is a graph on $n$ vertices. Then for any integer $c \geq 2n$, the independence number of the exponential graph satisfies
    \[
        \alpha(\E_c(H)) \leq n c^{n-1}.
    \]
\end{lem}
\begin{proof}
    First, observe that if $H$ and $H'$ have the same vertex set, and if $H \subseteq H'$, then $\E_c(H') \subseteq \E_c(H)$. This is because every pair of co-proper maps on $H'$ are also co-proper on $H$, as the edge set of $H$ is a subset of that of $H'$. Thus, if we let $K_n^\circ$ denote the complete graph on $n$ vertices where every vertex has a loop, then we see that $\E_c(K_n^\circ) \subseteq \E_c(H)$, so it suffices to upper-bound $\alpha(\E_c(K_n^\circ))$.
    
    By definition, the vertex set of $\E_c(K_n^\circ)$ is the set of maps $[c]^n$, and two vertices are adjacent in $\E_c(K_n^\circ)$ if and only if their images are disjoint.  Let $I$ be an independent set of $\E_c(K_n^\circ)$ and let $\F = \{ \im(\phi) \mid \phi \in I\}$. We can partition $\F$ into layers $\F_\ell$ according to the size of $\im(\phi)$, where $\F_\ell = \F \cap \binom{[c]} {\ell}$. By virtue of the fact that $I$ is an independent set, every pair of images of elements in $I$ must intersect, so $\F_\ell$ is an intersecting family in $\binom{[c]}{\ell}$. As $\ell \le n \le c/2$, the Erd\H os--Ko--Rado theorem~\cite{ErKoRa} applies to give
    \[
    |\F_\ell| \le \binom{c-1} {\ell - 1}.
    \]
    It follows that if $a_\ell$ is the number of surjective maps $[n]\twoheadrightarrow [\ell]$, then
    \[
    |I| \le \sum_{\ell=1}^{n} |\F_\ell| \cdot a_\ell \le \sum _{\ell=1}^n \binom{c-1} {\ell - 1} a_\ell.
    \]
    The right hand side is exactly the number of maps in $[c]^n$ which contain $1$ in their image. The number of such maps is at most $n\cdot c^{n-1}$, since there are $n$ ways to pick a vertex to send to $1$ and at most $c^{n-1}$ ways to color the rest. Thus, $|I|\le nc^{n-1}$ for all independent sets $I$, as desired.
\end{proof}

We remark that for fixed $n$, Lemma~\ref{lem:exp-independence} is tight for $H=K_n^\circ$ up to an additive error of $O(c^{n-2})$.

\section{Robust Colors}

The main technical lemma of Shitov's argument shows that for every suited $c$-coloring $\Psi$ of $\E_c(H)$ there is a ``central" vertex $v\in V(H)$ for which the color of a map $\phi \in \E_c(H)$ must appear in a ball around $v$. We extend his lemma so that it applies to suited $(c+t)$-colorings of $\E_c(H)$ as well. Write $\overline{N}(v) = \{v\} \cup N(v)$ for the {\it closed neighborhood} of a vertex $v$.

\begin{Def}
        Given a suited $(c+t)$-coloring $\Psi$ of $\E_c(H)$ and a vertex $v\in V(H)$, we say that a primary color $b\in[c]$ is {\it $v$-robust} if for every $\phi \in \Psi^{-1}(b)$ there exists a vertex $w\in \overline{N}(v)$ with $\phi(w)=b$.
\end{Def}

Lemma~\ref{lem:robust} below can be thought of as an analogue of a stablility result for the Erd\H os--Ko--Rado theorem. Such stability results (see e.g.~\cite{DiFr, Fr}) say that for an intersecting family of $k$-subsets of an $n$-set with size close to the maximum $\binom{n-1} {k-1}$, there exists a particular element in all of the sets. 

An independent set in $\E_c(H)$ is a family of pairwise non-co-proper maps, and we think of ``co-proper maps" as an analogue of ``disjoint sets," and being non-co-proper as ``intersecting" on a particular edge. We will show that for every large independent set $I$ in $\E_c(H)$, there is a particular vertex $v$ for which most of the elements of $I$ intersect on an edge close to $v$.

\begin{lem}\label{lem:robust}
    If $H$ is a triangle-free graph (possibly with loops) on $n\ge 4$ vertices, $c\ge 16(nt+n^3)$, $\Psi$ is a suited $(c+t)$-coloring of $\E_c(H)$, and
    \[
    x = \sqrt[4]{(nt+n^3)c^3},
    \]
    then there exists a vertex $v \in V(H)$ such that at least $c-x$ primary colors are $v$-robust.
\end{lem}
\begin{proof}
For each $v\in V(H)$ and each primary color $b\in [c]$, let $I(v,b)$ be the set of maps $\phi\in\Psi^{-1}(b)$ for which $\phi(v)=b$. Since $b$ is a primary color and $\Psi$ is a suited coloring, every $\phi \in \Psi^{-1}(b)$ is in some $I(v,b)$.

We say that $I(v,b)$ is {\it large} if it contains more than $n^2 c^{n-2}$ elements. We first show that if $b$ is a primary color and $I(v,b)$ is large, then $b$ is $v$-robust. If not, there is some $\phi \in \Psi^{-1}(b)$ which does not take value $b$ on any neighbor of $v$. Since $\Psi^{-1}(b)$ is an independent set of $\E_c(H)$, $\phi$ is not co-proper with any element of $I(v,b)$. On the other hand, we can upper bound the number of $\psi \in I(v,b)$ which are not co-proper to $\phi$. Such a map must have $\psi(v)=b$ (because it is in $I(v,b)$) and $\psi(w)\in \im(\phi)$ for some $w\ne v$. This $w$ can be picked in at most $n$ ways, and $\psi(w)$ itself in at most $n$ ways since $|{\im(\phi)}| \le n$. The total number of ways to pick such a $\psi$ from $I(v,b)$ is at most $n^2c^{n-2}$. This is a contradiction if $|I(v,b)|>n^2 c^{n-2}$, so if $I(v,b)$ is large, then $b$ is $v$-robust.

Given a primary color $b$, define $V_b$ to be the set of vertices $v\in V(H)$ for which $I(v,b)$ is large. We claim that $V_b$ is always a clique of $H$. To see this, let $v,w$ be two distinct vertices of $V_b$. Since $I(v,b)$ is large and there are at most $nc^{n-2}$ maps in $I(v,b)$ which send two or more vertices to $b$, there exists some map $\phi \in I(v,b)$ for which $\phi^{-1}(b) = \{v\}$ exactly. Also, since $I(w,b)$ is large, $b$ is $w$-robust, and so $\phi$ must take the value $b$ on some element of $\overline{N}(w)$. It follows immediately that $v \in N(w)$. Since this needs to hold for every pair $v,w\in V_b$, the set $V_b$ forms a clique in $H$, as desired.

Since we assumed $H$ is triangle-free, it follows that $|V_b|\le 2$ for each $b$. It remains to show that there is a vertex $v$ such that $v\in V_b$ for at least $c-x$ primary colors $b$. Suppose otherwise.

Let $S$ be the set of maps $\phi \in [c]^n$ which have the property that $\phi(v)\ne b$ whenever $v\in V_b$. Let $s(v)$ denote the number of primary colors $b$ for which $v \not \in V_b$, so $|S| = \prod _v s(v)$. By the assumption that fewer than $c-x$ primary colors are $v$-robust, we see that $x  < s(v) \le c$ for all $v$. Also,
    \[
    \sum_{v \in V(H)} s(v) = nc - \sum_{b\in [c]} |V_b| \ge (n - 2) c.
    \]
    By a standard convexity argument, the product of $s(v)$ is minimized when their values are as far apart as possible under the above assumptions. Thus,
    \[
    |S| = \prod_{v\in V(H)} s(v) > x^{A} c^{n-A},
    \]
    where $A$ satisfies
    \[
    Ax+(n-A)c=(n-2)c,
    \]
    implying that $A= 2c / (c- x) \le 4$, since our assumption on $c$ implies $x \le \frac{c}{2}$. Therefore,
    \[
        |S| \geq x^A c^{n-A}> x^{4} c^{n - 4} \geq (nt+n^3)c^{n-1}.
    \]
    
    On the other hand, we know that no $\phi \in S$ is in any large $I(v,b)$ for any primary color $b$, since such a $\phi$ does not take value $b$ on $V_b$. Thus, $S$ lies in the union of the secondary color classes and the small sets $I(v,b)$ for primary colors $b$. The sizes of the latter we bound by $n^2c^{n-2}$, and the former by $nc^{n-1}$ by Lemma~\ref{lem:exp-independence}, whereby
    \[
    |S| \le nc \cdot (n^2c^{n-2}) + t\cdot (nc^{n-1}) = (nt + n^3)c^{n-1}.
    \]
    We have arrived at a contradiction, so there exists a vertex $v$ which lies in at least $c-x$ of the sets $V_b$. Each of these primary colors $b$ is $v$-robust, as desired.
\end{proof}

\section{The Construction}
If $G$ and $H$ are finite simple graphs, their \emph{strong product} $G \boxtimes H$ is a simple graph on vertex set $V(G) \times V(H)$, where vertices $(g_1,h_1)$ and $(g_2,h_2)$ are adjacent if one of the following three conditions hold.
\[
    g_1 \sim g_2, h_1 \sim h_2 \qquad \text{ or }\qquad g_1 \sim g_2, h_1=h_2 \qquad\text{ or }\qquad g_1=g_2, h_1 \sim h_2.
\]

Let $G$ be a finite simple graph. We will be studying the exponential graphs of the two graphs $G^\circ$ and $G\boxtimes K_q$, for some $q\ge 2$. Note that there is a natural embedding $\iota:\E_c(G^\circ)\hookrightarrow \E_c(G\boxtimes K_q)$ where an element $\phi$ of $V(\E_c(G^\circ)) = [c]^{V(G)}$ is sent to the map $\phi^*:(g,i)\mapsto \phi(g)$ which ignores the $K_q$ coordinate.

\begin{lem}\label{lem:main}
    Fix a simple graph $G$ with $n\ge 4$ vertices and girth at least $6$, and let $\delta = \frac{1}{81n}$. If $q$ is sufficiently large in terms of $n$ and if $c = (3+10\delta)q$, then $\chi(\E_c(G\boxtimes K_q)) > (1+\delta)c$.
\end{lem}
\begin{proof}
Let $t=\delta c$. Suppose for the sake of contradiction that $\chi(\E_c(G \boxtimes K_q)) \leq (1+\delta)c= c+t$, so that by Lemma \ref{lem:suited} there is a suited $(c+t)$-coloring $\Psi$ of $\E_c(G\boxtimes K_q)$. Since $\E_c(G^\circ)$ is an induced subgraph of $\E_c(G\boxtimes K_q)$, $\Psi$ induces a suited $(c+t)$-coloring $\Psi^\circ$ on $\E_c(G^\circ)$.

With our choices of $t$ and $c$, the condition $c \ge 16(nt+n^3)$ holds if $q$ is sufficiently large. Also, $G$ has girth at least $6$ and in particular is triangle-free, so Lemma~\ref{lem:robust} applies to the graph $G^\circ$. Thus, there is a vertex $v\in V(G^\circ)$ such that at least $c-x$ primary colors of $\Psi^\circ$ are $v$-robust, where 
\[
    x =  \sqrt[4]{(nt+n^3)c^3}=((\delta n)^{1/4}+o(1))c
\]
as $q\rightarrow \infty$. By our choice of $\delta$, this means that $x = (\frac13 +o(1)) c$. We find that
\begin{align*}
    &c-x  = (\tfrac23 +o(1)) c = (2 + \tfrac{20}3 \delta + o(1))q, \\
    &2q + t + 1  = (2 + 3\delta + 10 \delta^2 +o(1))q.
\end{align*}
Observe that $10\delta <\frac {11}3$, so $c-x \ge 2q + t  +1$ for $q$ large enough. Thus, there exist $t+1$ primary colors $\sigma_1,\ldots, \sigma_{t+1}\not \in \{1,\ldots, 2q\}$ which are $v$-robust in the coloring $\Psi^\circ$.

We next pick a set $\M = \{\mu_{q+1},\ldots, \mu_{c}\}$ of vertices in $\E_c(G\boxtimes K_q)$. They are defined by
\[
\mu_r(g, i) = 
\begin{cases}
    i & \text{if } \dist(v, g) \in \{ 0, 2\},  \\ 
    q + i & \text{if } \dist(v, g) =1,    \\
    r & \text{otherwise}.
\end{cases}
\]
We claim that if $r\neq r'$, then $\mu_r$ and $\mu_{r'}$ are co-proper. To see this, suppose that $\mu_r(g,i)=\mu_{r'}(h,j)$ for adjacent $(g,i),(h,j)$. Notice that this is impossible if $i \neq j$, so $i=j$, implying that $g \sim h$. Since the girth of $G$ is at least $6$, there are no edges $(g,h)\in E(G)$ for which $\dist(v,g)$ and $\dist(v,h)$ are both at most $2$ and have the same parity. Thus, without loss of generality, $g$ has distance at least $3$ from $v$. This implies $\mu_r(g,i)=r$. The only way $\mu_{r'}(h,j)=r$ is if $\dist(v,h)=1$ and $q+j=r$, but then $\dist(g,h)\geq 2$ by the triangle inequality, a contradiction. Therefore, we conclude that $\mu_r$ and $\mu_{r'}$ are co-proper whenever $r\ne r'$, so $\M$ forms a clique of size $c-q$ in $\E_c(G\boxtimes K_q)$.

Since we chose $t = (3\delta +10 \delta^2)q$, if $q$ is sufficiently large, we have 
\[
c-3q = 10\delta q \ge (9 \delta+30 \delta^2)q+2= 3t + 2
\]
and therefore $c-3q-2t-1 \ge t+1$. In particular, at least $t+1$ of the colors $\{\Psi(\mu_r)\}_{r=q+1}^c$ do not lie in the union $\{1,\ldots, 2q\}\cup \{\sigma_1,\ldots, \sigma_{t+1}\} \cup \{c+1,\ldots, c+t\}$.

Let $\M' = \{\mu_{r_1},\ldots, \mu_{r_{t+1}}\}$ be a set of $t+1$ vertices of $\M$ with colors not among $\{1,\ldots, 2q\}\cup \{\sigma_1,\ldots, \sigma_{t+1}\} \cup \{c+1,\ldots, c+t\}$. Since $\Psi$ is a suited coloring and $\im(\mu_{r_s}) = \{1,\ldots, 2q\}\cup \{r_s\}$, it follows that $\Psi(\mu_{r_s}) =r_s$ for each $1 \leq s\leq t+1$.

We define a set $\N$ of $t+1$ other vertices $\nu_{1}, \ldots, \nu_{t+1}$ in $\E_c(G\boxtimes K_q)$, by
\[
\nu_{s}(g, i) = 
\begin{cases} 
    r_s & \text{if } \dist(v, g) \le 1,  \\ 
    \sigma_s & \text{otherwise}.    
\end{cases}
\]

Recall that we chose $\{r_1,\ldots, r_{t+1}\}$ to be $t+1$ distinct colors disjoint from $\{\sigma_1,\ldots,\sigma_{t+1}\}$. Therefore, $\nu_s$ and $\nu_{s'}$ have disjoint images when $s \neq s'$, so $\N$ forms a clique of size $t+1$ in $\E_c(G \boxtimes K_q)$. Also, the maps $\nu_s$ are constant on the $K_q$ coordinate, so they lie in the image of the embedding $\iota:\E_c(G^\circ)\hookrightarrow \E_c(G\boxtimes K_q)$ and correspond to vertices of $\E_c(G^\circ)$. We chose $\sigma_s$ to be $v$-robust in $\Psi^\circ$, so if $\Psi^\circ(\iota^{-1}(\nu_s))=\sigma_s$, then $\iota^{-1}(\nu_s)(g)=\sigma_s$ for some $g\in V(G^\circ)$ with $\dist(v,g) \leq 1$. This contradicts the definition of $\nu_s$, so $\Psi(\nu_s) = \Psi^\circ(\iota^{-1}(\nu_s))\ne \sigma_s$. 

Therefore, by the suitedness of $\Psi$, $\Psi(\nu_s) \in \{r_s\} \cup \{c+1,\ldots, c+t\}$ for each $1 \leq s \leq t+1$. There are only $t$ secondary colors to use for the $t+1$ vertices of this clique, which implies that for some $s$, $\Psi(\nu_s) = r_s$. But $\Psi(\mu_{r_s}) = r_s$ as well, and $\mu_{r_s}$ and $\nu_s$ are co-proper in $\E_c(G\boxtimes K_q)$. This is the desired contradiction, completing the proof that $\chi(\E_c(G\boxtimes K_q)) > (1+\delta)c$.
\end{proof}

All that remains is to find a $G$ with large fractional chromatic number which satisfies the conditions of Lemma~\ref{lem:main}. A famous probabilistic argument of Erd\H os~\cite{Er} shows that there exist graphs of arbitrarily large girth and fractional chromatic number; in order to obtain as large a value of $\delta$ as possible, we choose the parameters as follows.
\begin{lem}\label{lem:exists-graph}
    There exists a simple graph $G$ on at most $2\times 10^6$  vertices with $\girth(G) \geq 6$ and $\chi_f(G) \geq 3.1$. 
\end{lem}
\begin{proof}
    Let $n=2\times 10^6$ and $p=8\times 10^{-6}$, and let $G_0 \sim G(n,p)$ be an Erd\H os--R\'enyi random graph with these parameters. Let $X$ denote the number of cycles of length at most $5$ in $G_0$; then we can compute
    \[
        \mathbb E[X]\leq \frac{n^3 p^3}{6}+\frac{n^4 p^4}{8}+\frac{n^5 p^5}{10} \leq 115000=:t.
    \]
    Therefore, by Markov's inequality, with probability at least $1/2$, $G_0$ will contain at most $2t$ cycles of length at most $5$. Next, let $k=570000$, and observe that
    \[
        \binom nk (1-p)^{\binom k2} < \frac 14.
    \]
    Therefore, with probability at least $3/4$, $G_0$ will contain no independent set of size $k$. Thus, there exists a specific graph $G_1$ on $n$ vertices with at most $2t$ cycles of length at most $5$ and $\alpha(G_1) \leq k$. We delete one vertex from each cycle of length at most $5$ in $G_1$ to obtain a new graph $G$ on at least $n-2t$ vertices with $\girth(G) \geq 6$ and $\alpha(G) \leq \alpha(G_1) \leq k$. Moreover,
    \[
        \chi_f(G) \geq \frac{|V(G)|}{\alpha(G)} \geq \frac{n-2t}{k} \geq 3.1,
    \]
    as desired.
\end{proof}

\begin{proof}[Proof of Theorem~\ref{thm:main}.] Let $G$ be the graph from Lemma \ref{lem:exists-graph} on $n\leq 2\times 10^6$ vertices. Let $\delta=1/(81n) \geq 10^{-9}$, let $q$ be sufficiently large so that Lemma~\ref{lem:main} applies, and let $c=(3+10\delta)q$. We will only prove Theorem \ref{thm:main} for $c$ of this form; it can be proved for all $c$ large enough by rounding off to the nearest such value.

It is easy to see that
\[
\chi(G\boxtimes K_q)\geq \chi_f(G)\chi_f(K_q) \ge 3.1q > (1 + \delta)c.
\]
By Lemma~\ref{lem:main}, we also know that $\chi(\E_c(G\boxtimes K_q)) > (1+\delta)c$. Since $G \boxtimes K_q$ is not $c$-colorable, we see that $\E_c(G \boxtimes K_q)$ is simple, as is $G \boxtimes K_q$.

On the other hand, Lemma~\ref{lem:c-colorable} shows that $\chi((G\boxtimes K_q) \times \E_c(G\boxtimes K_q)) \le c$, which proves the theorem for some $\delta \ge 10^{-9}$. 
\end{proof}

\paragraph{Acknowledgments.} We would like to thank Nitya Mani for interesting discussions on Shitov's proof, and Lisa Sauermann for many helpful comments on the early drafts of this paper.

\end{document}